\documentclass[10pt,a4paper]{amsart}
\pdfoutput=1
\usepackage[latin1]{inputenc}
\usepackage{amsmath,amssymb}
\usepackage{mathtools}
\usepackage{esint}
\usepackage{color}
\usepackage{hyperref}

\DeclareMathOperator{\dive}{div}

\DeclareMathOperator{\supp}{supp}

\DeclareMathOperator{\sign}{sign}

\newcommand*{\dx}[1]{\, d #1}
\newcommand*{\R}{\mathbb{R}}

\newcommand*{\eps}{\varepsilon}
\newcommand*{\doo}{\partial}
\newcommand*{\ol}[1]{\overline{#1}}
\newcommand*{\gradeps}{\left( |\nabla u_\eps|^2 + \eps \right)}
\newcommand{\abs}[1]{\left| #1 \right|}
\newcommand{\aabs}[1]{\left\| #1 \right\|}

\newtheorem{theorem}{Theorem}
\newtheorem{defin}[theorem]{Definition}
\newtheorem{lemma}[theorem]{Lemma}

\newtheorem{remark}[theorem]{Remark}

\numberwithin{equation}{section}
\numberwithin{theorem}{section}

\title[Calder\'on problem for the $p$-Laplacian]{Calder\'on problem for the $p$-Laplacian: First order derivative of conductivity on the boundary}
\author{Tommi Brander}
\address{Department of Mathematics and Statistics, P.O.Box 35 (MaD) FI-40014 University of Jyv\"askyl\"a, Finland}
\email{tommi.o.brander@jyu.fi}
\date{\today}

\begin{document}

\begin{abstract}
We recover the gradient of a scalar conductivity defined on a smooth bounded open set in $\mathbb{R}^d$ from the Dirichlet to Neumann map arising from the $p$-Laplace equation. For any boundary point we recover the gradient using Dirichlet data supported on an arbitrarily small neighbourhood of the boundary point. We use a Rellich-type identity in the proof. Our results are new when $p \neq 2$. In the $p = 2$ case boundary determination plays a role in several methods for recovering the conductivity in the interior.
\end{abstract}



\maketitle 


\section{Introduction}
Throughout the article we assume $1 < p < \infty$, unless explicitly written otherwise.
We investigate a generalisation of the Calder\'on problem to the case of the $p$-Laplace equation:
Given a bounded open set $\Omega \subset \R^d$ and a bounded conductivity $\gamma > 0$ on $\ol\Omega$ we have the Dirichlet problem
\begin{equation} \label{eq:p-laplace}
\begin{cases}
\Delta_p^\gamma(u) = \dive \left( \gamma(x) |\nabla u|^{p-2} \nabla u \right) = 0 &\text{in $\Omega$,} \\
u = v &\text{on $\doo \Omega$}.
\end{cases}
\end{equation}
For $v \in W^{1,p}(\Omega)$ there exists a unique weak solution that satisfies $\| u \|_{W^{1,p}} \leq C \| v \|_{W^{1,p}}$, where the constant $C$ does not depend on $v$.
This is completely standard, see e.g. \cite[Proposition A.1]{Salo:Zhong:2012}.
The equation arises from the variational problem of minimising the energy $\int_\Omega \gamma |\nabla u|^p \dx{x}$.

We next define the Dirichlet to Neumann map (hereafter DN map).
We use the notation $X'$ for the dual space of continuous linear functionals on a Banach space $X$.
In the $p=2$ case the DN map gives the stationary measurements of electrical current for given voltage $v$.
\begin{defin}[Dirichlet to Neumann map]\label{definition:DN_map}
Suppose  $\Omega \subset \R^d$ is bounded open $C^1$-set and $0 < \gamma_0 < \gamma \in L^\infty(\overline{\Omega})$ for some constant~$\gamma_0$.
The weak DN map
\begin{equation*}
\Lambda_\gamma^w \colon W^{1,p}(\Omega)/W^{1,p}_0(\Omega) \to \left(W^{1,p}(\Omega)/W^{1,p}_0(\Omega)\right)'
\end{equation*}
is defined by
\begin{equation}\label{eq:weak_DN}
\left\langle \Lambda_\gamma^w(v), g \right\rangle = \int_\Omega \gamma |\nabla u|^{p-2} \nabla u \cdot \nabla \tilde g \dx{x},
\end{equation}
where $u$ solves the boundary value problem~\eqref{eq:p-laplace} with boundary values~$v$, and $\tilde g \in W^{1,p}(\Omega)$ with trace~$g$ on the boundary.
The strong DN map is defined pointwise on $\doo \Omega$ by the expression
\begin{equation}\label{eq:strong_DN}
\Lambda_\gamma^s (v)(x_0) = \gamma(x_0) |\nabla u(x_0)|^{p-2}\nabla u(x_0) \cdot \nu(x_0)
\end{equation}
when it is well-defined.
\end{defin}
We also write $\Lambda_\gamma (v)$ when the function $v$ is continuous and defined on a superset of $\doo \Omega$.
With sufficient regularity we can recover the values of the strong DN map from the weak DN map (see lemma~\ref{lemma:weak_DN_to_strong}).

Given knowledge of the DN map we determine $\nabla \gamma|_{\doo \Omega}$, which in the $p=2$ case stands for the gradient of conductivity on the boundary.
\begin{theorem}\label{thm:main}
Suppose $\Omega \subset \R^d$, $d > 1$, is a bounded open $C^{2,\beta}$ set for some $0 < \beta <1$ and the conductivities $\gamma_1$ and $\gamma_2$ are bounded from below by a positive constant and continuously differentiable with H{\"o}lder-continuous derivatives in~$\ol{\Omega}$.
If $\Lambda_{\gamma_1}^w = \Lambda_{\gamma_2}^w$, then $\nabla \gamma_1|_{\doo \Omega} = \nabla \gamma_2|_{\doo \Omega}$.
\end{theorem}
We use an explicit sequence of boundary values to reconstruct $\nabla \gamma (x_0)$ at a boundary point $x_0$.
The sequence is supported in arbitrarily small neighbourhood of $x_0$.
The boundary values we use were first introduced by Wolff~\cite{Wolff:2007} and used by Salo and Zhong~\cite{Salo:Zhong:2012} to recover the conductivity on the boundary.

Salo and Zhong~\cite{Salo:Zhong:2012} show that for all boundary points $x_0 \in \doo \Omega$ there is a sequence of solutions $(u_M)_M$ such that
\begin{equation}
M^{n-1}N^{1-p}\int_\Omega g(x) |\nabla u_M|^p \dx{x} \to c_p g(x_0)
\end{equation}
as $M, N(M) \to \infty$, with explicitly computable constant $c_p$ (see equation~\eqref{eq:c_p}).
They recover the integrals with $g = \gamma$ from the DN map.
We recover the integrals with $g = \alpha \cdot \nabla \gamma$ from a Rellich-type identity, theorem~\ref{thm:rellich}, for arbitrary direction~$\alpha \in \R^d$.

The Calder\'on problem was first introduced in~\cite{Calderon:1980}. For a review see~\cite{Uhlmann:2009}.

There have been several boundary determination results in the $p = 2$ case. Boundary uniqueness for derivatives of the conductivity was first proven by Kohn and Vogelius~\cite{Kohn:Vogelius:1984}.
Sylvester and Uhlmann in ~\cite{Sylvester:Uhlmann:1988} recovered, for smooth conductivity in smooth domain, conductivity and all its derivatives on the boundary by considering $\Lambda_\gamma$ as a pseudodifferential operator.
Nachman~\cite{Nachman:1996} recovered $\gamma|_{\doo \Omega}$ in Lipschitz domain when $\gamma \in W^{1,q}$, $q > d$, and its first derivative when $\gamma \in W^{2,q}$, $q > d/2$.
Alessandrini~\cite{Alessandrini:1990} used singular solutions with singularity near boundary to recover all derivatives of $\gamma|_{\doo \Omega}$ with less regularity assumptions on $\gamma$.
There have been several local boundary determination results, e.g.~\cite{Nakamura:Tanuma:2003,Kang:Yun:2002}.
Other reasonably recent boundary determination results include~\cite{Alessandrini:Gaburro:2009,Brown:2001}.

A Rellich identity was used by Brown, Garcia and Zhang~\cite[appendix]{Garcia:Zhang:2012} to recover the gradient of conductivity on the boundary in the $p=2$ case.
The Rellich identity was, to the best of our knowledge, introduced in~\cite{Rellich:1940}.

Several results (e.g.~\cite{Alessandrini:1990}) rely on investigating the difference $\Lambda_{\gamma_1} - \Lambda_{\gamma_2}$ of DN~maps at different conductivities.
This is difficult in the present setting due to non-linearity of the $p$-Laplace equation.
We use Rellich identity, theorem~\ref{thm:rellich}, to avoid this problem.

Electrical impedance tomography has applications in medical (e.g.~\cite{Barber:Brown:Seagar:1985}) and industrial imaging (e.g.~\cite{Karhunen:2013}), and geophysics and environmental sciences; see e.g.\ the review~\cite{Borcea:2002} and references therein.
There are practical numerical algorithms for boundary determination in the $p=2$ case, e.g.~\cite{Nakamura:Siltanen:Tamura:Wang:2005}.
Boundary determination is used in recovering the conductivity in the interior, e.g.~\cite{Haberman:Tataru:2013}.
For example, the algorithm in~\cite{Siltanen:Tamminen:2011} uses the values of conductivity and its first derivative on the boundary to extend the conductivity, thence applying to conductivities that are not constant on the boundary.

For more on the $p$-Laplace equation see e.g.~\cite{Lindqvist:2006,Heinonen:Kilpelainen:Martio:1993,D'Onofrio:Iwaniec:2005}.
The equation has applications in e.g.\ image processing~\cite{Kuijper:2007}, fluid mechanics~\cite{Aronsson:Janfalk:1992}, plastic moulding~\cite{Aronsson:1996}, and modelling of sand--piles~\cite{Aronsson:Evans:Wu:1996}.

The Calder\'on problem for the $p$-Laplace equation was first introduced in~\cite{Salo:Zhong:2012}.
The authors consider the real and the complex case separately.
In the complex case they define $p$-harmonic versions of complex geometrical optics solutions to create highly oscillating functions focused around a given boundary point and use them as Dirichlet data, thus recovering conductivity at the boundary point in question.
In the real case they replace the $p$-harmonic CGO solutions with real-valued functions having similar behaviour.
The real-valued functions were originally introduced by Wolff~\cite{Wolff:2007}.

Further progress on the $p$-Calder\'on problem was made by the author, Kar and Salo~\cite{Brander:Kar:Salo:2014}.
In the article we show that we can detect the convex hulls of inclusions, which are regions of significantly higher or lower conductivity, from the DN map.
Our main tools are the $p$-harmonic functions of Wolff and a monotonicity inequality.

Hauer~\cite{Hauer:2014} has investigated the DN map related to the $p$-Laplace equation.

One method of investigating the Calder\'on-type inverse problems for non-linear equations is based on studying the G\^ateaux derivatives of the map $\Lambda_\gamma$ at constant boundary values $a$.
In our case this does not work~\cite[appendix]{Salo:Zhong:2012}:
\begin{equation}
\Lambda_\gamma (a+tf) = t^{p-1}\Lambda_\gamma(f)
\end{equation}
for positive $t$.
For $p < 2$ the G\^ateaux derivates do not exist and for $p > 2$ the higher derivates fail to exist, or vanish, though one of them might equal $\Lambda_\gamma (f)$.
Hence nonlinear methods are necessary.

Other inverse problems related to nonlinear equations similar to the $p$-Laplace equation have been investigated before; the $1$-Laplace equation is used in current density imaging, e.g.~\cite{Kim:Kwon:Seo:Yoon:2002,Nachman:Tamasan:Timonov:2009,Hoell:Moradifan:Nachman:2013,Henkelman:Joy:Scott:1989,Armstrong:Henkelman:Joy:Scott:1991}, and the $0$-Laplace equation is related to ultrasound modulated electrical impedance tomography, e.g.~\cite{Bal:2012,Ammari:Bonnetier:Capdebocsg:Fink:Tanter:2008,Bal:Schotland:2010,Gebauer:Scherzer:2008}.

In section~\ref{sec:prelim} we introduce our notation and state lemmata.
In section~\ref{sec:salo_zhong} we restate boundary determination results we use in this article.
In the final section~\ref{section:proof} we reconstruct $\nabla \gamma|_{\doo \Omega}$ from the DN~map.

\section{Preliminaries}\label{sec:prelim}
We use the following notation:
We write $\doo_\alpha = \alpha \cdot \nabla$ for vectors $\alpha$.
On the $C^1$ boundary $\doo \Omega$ of an open set $\Omega$ we decompose the gradient of a function $f$ defined in $\ol{\Omega}$ as $\nabla f(x_0) = \nabla_\nu f(x_0) + \nabla_T f(x_0)$, where $\nabla_\nu$ is the normal component of the derivative, i.e.\ the orthogonal projection of the gradient to the normal space of the boundary, and $\nabla_T$ is the tangential derivative, i.e.\ the orthogonal projection of the gradient on the tangent space of the boundary.
We have the following identity: $|\nabla_\nu f|^2 = (\doo_\nu f)^2$.
We denote the outer unit normal by $\nu$.

We write the H{\"o}lder seminorm of order $k$ as $|f|_{k,\beta}$ and the characteristic function of a set $A$ as
\begin{equation*}
\chi_A(x) =
\begin{cases}
1 \text{ when } x \in A \\
0 \text{ when } x \notin A.
\end{cases}
\end{equation*}

We need the following elementary inequality:
\begin{lemma}\label{lemma:A.3}
Suppose $0 \leq A,B \in \R$ and $2 \leq q < \infty$.
Then for some $C > 0$ we have
\begin{equation}
|A^{q-1} - B^{q-1}| \leq C \left( A + B \right)^{q-2} | A - B|
\end{equation}
\end{lemma}
\begin{proof}
By fundamental theorem of calculus
\begin{equation*}
|A^{q-1} - B^{q-1}| = (q-1)\abs{A-B}\abs{\int_0^1 \left(A + t(B-A)\right)^{q-2} \dx{t}},
\end{equation*}
from which the estimate follows by observing that $ 0 \leq A+t(B-A) \leq A+B$.
\end{proof}

It is well-known that the solutions of the $p$-Laplace equation are in $C^{1,\beta}$. For proof of the following lemma see e.g.~\cite{Lieberman:1988}.
\begin{lemma}[Regularity result] \label{lemma:regularity}
Suppose $\Omega$ is a bounded open $C^{1,\beta_1}$ set with $0 < \beta_1 \leq 1$,
and suppose the conductivity $0 < \gamma \in C^{0,\beta_2}(\ol{\Omega})$ is bounded from above and away from zero.
Consider the weighted $p$-Laplace equation~\eqref{eq:p-laplace} with boundary values $v \in C^{1,\beta_1}(\doo \Omega)$.
Then the solution~$u$ of the weighted $p$-Laplace equation~\eqref{eq:p-laplace} is in $C^{1,\beta_3}(\ol{\Omega})$ for some $\beta_3 > 0$.
\end{lemma}
In general the solutions are not twice continuously differentiable~\cite{Iwaniec:Manfredi:1989}.

To establish Rellich identity, theorem~\ref{thm:rellich}, we use the following interpolation lemma similar to~\cite[lemma A.1]{Liimatainen:Salo:2012}:
\begin{lemma}[Interpolation lemma] \label{lemma:interpolation}
Suppose $\Omega$ is a bounded open set in $\R^d$ satisfying the measure density condition; that is, suppose there exist $\delta_0, C > 0$ such that for all $0 < \delta < \delta_0$ and for all $y \in \ol\Omega$ we have
\begin{equation}\label{eq:mdc}
|B(y,\delta)| \leq C |B(y,\delta) \cap \ol\Omega|.
\end{equation}
Also suppose $f \in C^\beta(\ol{\Omega})$ for some $\beta > 0$.
Take $1 \leq p < \infty$.
Let $M > 0$ be a constant such that $|f|_{0,\beta} \leq M$.
Then
\begin{equation}
\| f \|_{L^\infty(\ol{\Omega})} \leq C_{d,p,\beta,\Omega}M^{\frac{d}{d+\beta p}} \|f\|_{L^p(\Omega)}^\frac{\beta p}{d+ \beta p}.
\end{equation}
\end{lemma}
The measure density condition~\eqref{eq:mdc} is satisfied in a wide variety of situations~\cite{Hajlasz:Koskela:Tuominen:2008}.
In particular, domains that admit a Sobolev extension theorem satisfy the condition~\cite[proposition 1]{Hajlasz:Koskela:Tuominen:2008}.
Domains of class $C^{1,\alpha}$, and more generally Lipschitz domains, certainly admit a Sobolev extension theorem~\cite[theorem 12]{Calderon:1961}, \cite[chapter VI, section 3.1, theorem 5]{Stein:1970}.
A bounded Lipschitz set has a finite number of connected components, and thus also satisfies the condition.

The proof of the interpolation lemma is almost the same as in~\cite{{Liimatainen:Salo:2012}}.
\begin{proof}
Let $y \in \ol{\Omega}$ and write $B_\delta = B(y,\delta)$.
For $\delta > 0$ we have
\begin{align*}
\|f\|_{L^p(\ol{\Omega}\cap B_\delta)} &\geq \|f(y)\|_{L^p(\ol{\Omega}\cap B_\delta)} - \|f(x) - f(y)\|_{L^p(\ol{\Omega}\cap B_\delta)} \\
&= |f(y)|\left|\ol{\Omega}\cap B_\delta\right|^{1/p} - \left( \int_{\ol{\Omega}\cap B_\delta} |f(x) - f(y)|^p \dx{x} \right)^{1/p} \\
&\geq c_{d,p,\Omega} \delta^{d/p}|f(y)| - |f|_{0,\beta}\left(\int_{\ol{\Omega}\cap B_\delta} |x-y|^{\beta p} \dx{x}\right)^{1/p} \\
&\geq c_{d,p,\Omega} \delta^{d/p}|f(y)| - C_{d,p,\beta} \delta^{\beta + d/p}M,
\end{align*}
whereby it follows that
\begin{equation}
\| f \|_{L^\infty(\ol{\Omega})} \leq C_{d,p,\Omega}\delta^{-d/p} \|f\|_{L^p(\ol{\Omega})} + C_{d,p,\beta}\delta^\beta M.
\end{equation}
Choose
\begin{equation}
\delta = \left( \frac{\|f\|_{L^p(\ol\Omega)}}{M} \right)^\frac{p}{d + \beta p}.
\end{equation}
\end{proof}

In our proofs we use the $\eps$-perturbed $p$-Laplace equation.
Let $\eps > 0$ and define $u_\eps$ as the solution of the boundary value problem
\begin{equation} \label{eq:eps-laplace}
\begin{cases}
\dive \left( \gamma(x) \left(|\nabla u_\eps|^2 + \eps\right)^\frac{p-2}{2} \nabla u_\eps \right) = 0 &\text{in $\Omega$,} \\
u_\eps = v &\text{on $\doo \Omega$}.
\end{cases}
\end{equation}
By calculus of variations there exists a unique solution to~\eqref{eq:eps-laplace} when $\Omega$ is a bounded open set and $v \in W^{1,p}(\Omega)$.
The energy corresponding to~\eqref{eq:eps-laplace} is
\begin{equation}\label{eq:eps_energy}
\int_\Omega \gamma \left(|\nabla u_\eps|^2 + \eps\right)^{p/2} \dx{x}.
\end{equation}

\begin{lemma}[Regularity result for perturbed equation]\label{lemma:eps_regularity}
Suppose $\Omega$ is a bounded open set with $C^{1,\beta_1}$ boundary.
Suppose that boundary values $v \in C^{1,\beta_1}(\doo \Omega)$ and conductivity $\gamma$ is H{\"o}lder-continuous.
Then $u_\eps \in C^{1,\beta_2}(\ol \Omega)$ with some $0 < \beta_2 < 1$ independent of $\eps$ and the norm $|u_\eps|_{1,\beta_2}$ has an upper bound independent of~$\eps$.

Furthermore, suppose that $\Omega$ has $C^{2,\beta_3}$ boundary, $v \in C^{2,\beta_3}(\doo \Omega)$ and that $\nabla \gamma$ is H{\"o}lder-continous.
Then $ u_\eps \in C^{2,\beta_4}(\ol{\Omega})$ for some $0 < \beta_4 < 1$, which depends on~$\eps$.
\end{lemma}
\begin{proof}
That $\nabla u_\eps$ are H{\"o}lder-continuous with $\beta_2$ and the corresponding norm independent of $\eps$ follows from~\cite{Lieberman:1988}.
The $C^2$--regularity is also standard; see for example~\cite[part 4, section 8]{Ladyzhenskaya:Ural'tseva:1968} or~\cite[section 15.5]{Gilbarg:Trudinger:1983}.
\end{proof}

\begin{lemma}[Convergence of the perturbed equations]\label{lemma:perturbed_convergence}
Suppose $\Omega$ is a bounded open set with $C^{1,\beta}$ boundary, $0 < \beta < 1$, and boundary values $v$ are in $C^{1,\beta}$.
Also assume that the conductivity $\gamma$ is H\"older continuous.
Then $u_\eps \to u$ in $C^{1}(\ol \Omega)$.
\end{lemma}
\begin{proof}
We first show that the energy~\eqref{eq:eps_energy} of the perturbed equation~\eqref{eq:eps-laplace} converges to the energy of the non-perturbed equation~\eqref{eq:p-laplace}:
The solution $u$ minimises energy, and so we have
\begin{align*}
&\int_\Omega \gamma |\nabla u|^{p}\dx{x}\\
\leq &\int_\Omega \gamma |\nabla u_\eps|^{p}\dx{x}\\
\leq &\int_\Omega \gamma \left(|\nabla u_\eps|^2+\eps\right)^{p/2}\dx{x} \\
\leq &\int_\Omega \gamma \left(|\nabla u|^2+\eps\right)^{p/2}\dx{x} \\
\leq &\int_\Omega \gamma |\nabla u|^{p}\dx{x} + C
\begin{cases}
\int_\Omega \eps \left( 2|\nabla u|^2 + \eps \right)^{-1+p/2}\dx{x} &\text{ when } p > 2 \\
\int_\Omega \eps^{p/2} \dx{x} &\text{ when } p \leq 2,
\end{cases}
\end{align*}
where the latter integrals vanish as $\eps \to 0$ since $|\nabla u|$ is bounded (lemma~\ref{lemma:regularity}).
We used the fact that $u_\eps$ minimises the perturbed energy~\eqref{eq:eps_energy}, and also lemma~\ref{lemma:A.3} with $q = 1+p/2 > 2$ when $p>2$, and the fact that for non-negative numbers $(a+b)^q \leq a^q + b^q$ when $0 \leq q = p/2 \leq 1$, in the final inequality.

We established that the energy of $u_\eps$ converges to the energy of $u$ as $\eps \to 0$.
Since $\aabs{u_\eps}_{L^\infty}$ are uniformly bounded, it follows that $u_\eps$ are bounded in $W^{1,p}$ with weight.
This implies that the sequence has a weakly converging subsequence with limit that has lesser or equal energy than $u$.
But $u$ uniquely minimises energy, so $u_\eps \to u$ weakly in weighted $W^{1,p}$.
Since the weighted $L^p$ norms of $\nabla u_\eps$ converge to weighted $L^p$ norm of $\nabla u$, and the sequence converges weakly, we get that $\nabla u_\eps \to \nabla u$ strongly in weighted $L^{p}$ space; this is the Radon-Riesz property proven in~\cite{Radon:1913,Riesz:1928a,Riesz:1928b}.
Since the weighted and standard norms are equivalent, $\nabla u_\eps \to \nabla u$ strongly in standard $L^p$.

Since $\nabla u_\eps$ are uniformly H{\"o}lder in $\ol{\Omega}$ (lemma~\ref{lemma:eps_regularity}) we can use the interpolation lemma~\ref{lemma:interpolation} to deduce convergence of $\nabla u_\eps$ to $\nabla u$ in $C(\ol{\Omega})$.
The same argument holds for $u_\eps$, since by Friedrichs-Poincar\'e inequality $u_\eps \to u$ strongly in $L^p$.
\end{proof}

\section{Previous results on boundary determination for $p$-Laplacian}\label{sec:salo_zhong}
We use several results in the paper~\cite{Salo:Zhong:2012} and restate them here for convenience.

Suppose $\Omega \subset \R^d$ is an open bounded $C^1$ set and suppose $\rho \in C^1(\R^d)$ is its boundary defining function;
that is,
\begin{equation*}\Omega = \{ x \in \R^d ; \rho (x) > 0 \},\quad \doo \Omega = \{ x \in \R^d ; \rho (x) = 0 \}
\end{equation*}
and $\nabla \rho \neq 0$ on $\doo \Omega$.
By translation we may assume that we are trying to recover $\nabla \gamma$ at the origin $0 \in \doo \Omega$, and by rotation and scaling we may assume that
\begin{equation*}
-\nabla \rho (0) = \nu(0) = -e_d,
\end{equation*}
the $d$th coordinate vector.
We define the map $f \colon \Omega \to \R^d_+$ by
\begin{equation*}
f(x',x_d) = (x',\rho(x))
\end{equation*}
for $x = (x',x_d)$.
Since $\rho$ is a $C^1$ function, the map is close to the identity near the origin, and hence invertible in some neighbourhood of the origin.

We now introduce the $p$-harmonic oscillating functions of Wolff~\cite{Wolff:2007}.
\begin{lemma} \label{lemma:wolff}
Define the function $h \colon \R^d \to \R$ by $h(x)=e^{-x_d}a(x_1)$,
where $a \colon \R \to \R$ is the solution to the differential equation
\begin{equation}
a''(x_1) + V(a,a')a = 0
\end{equation}
with
\begin{equation}
V(a,a') = \frac{(2p-3)(a')^2 + (p-1)a^2}{(p-1)(a')^2 + a^2}.
\end{equation}
Then $\Delta_p^1(h)=0$ and the function $a$ is smooth and periodic with period $\lambda(p)$, so that $\int_0^\lambda a(x_1)\dx{x_1} = 0$.
\end{lemma}
For proof of the lemma see~\cite[lemma 3.1]{Salo:Zhong:2012}.

We define a smooth positive cutoff function $\zeta \in C^\infty_0(\R^d)$, such that $\zeta (x) = 1$ when $|x| < 1/2$ and $\supp \zeta \subset B(0,1)$.
We write
\begin{equation}\label{eq:v_M}
v_M(x) = h(Nf(x))\zeta(Mx),
\end{equation}
where $M$ and $N = N(M)$ are large positive numbers, $M = o(N)$.
We define $u_M$ to be the solutions of the initial value problem
\begin{equation}
\begin{cases}
\Delta_p^\gamma(u_M) = 0 &\text{in $\Omega$,} \\
u_M = v_M &\text{on $\doo \Omega$}.
\end{cases}
\end{equation}

In \cite[lemma 3.3]{Salo:Zhong:2012} Salo and Zhong show that
\begin{equation}
M^{d-1}N^{1-p}\int_\Omega \gamma |\nabla v_M|^p \dx{x} \to c_p \gamma(0)
\end{equation}
as $M \to \infty$.
The constant $c_p$ is explicit:
\begin{equation}\label{eq:c_p}
c_p = \frac{K}{p}\int_{\R^{d-1}}(\eta(x',0))^p\dx{x'}; \text{ with } K = \lambda^{-1} \int_0^\lambda \left(a^2(t) + (a'(t))^2\right)^{p/2}\dx{t}.
\end{equation}
The proof also holds when $\gamma$ is replaced by any other function continuous at~$0$, so as $M \to \infty$ we have
\begin{equation}
M^{d-1}N^{1-p}\int_\Omega \doo_\alpha \gamma |\nabla v_M|^p \dx{x} \to c_p \doo_\alpha \gamma(0).
\end{equation}
Using~\cite[lemma 3.4]{Salo:Zhong:2012} we get the following result:
\begin{lemma}\label{lemma:sz_main}
Suppose $\Omega \subset \R^d$ is a bounded open set with $C^1$ boundary, $x_0 \in \doo \Omega$, $d \geq 2$, and $g$ is continuous.
Then
\begin{equation}
M^{d-1}N^{1-p}\int_\Omega g(x) |\nabla u_M|^p \dx{x} \to c_p g(x_0)
\end{equation}
as $M \to \infty$.
\end{lemma}

\section{Proof of the main result} \label{section:proof}
In this section we establish that the quantity
\begin{equation}
\int_\Omega \doo_\alpha\gamma(x) |\nabla u|^p \dx{x}
\end{equation}
can be calculated from the measurements represented by the DN map, with the function~$u$ solving the weighted $p$-Laplace equation~\eqref{eq:p-laplace} with sufficiently smooth boundary values.
Then it follows from lemma~\ref{lemma:sz_main} that we can recover $\doo_\alpha \gamma(x_0)$ at any boundary point $x_0 \in \doo \Omega$.
The main result, theorem~\ref{thm:main}, immediately follows.

\begin{theorem}[Rellich identity] \label{thm:rellich}
Suppose $1<p<\infty$ and that $\Omega \subset \R^d$ is a bounded open set with $C^{2,\beta}$ boundary for some $0 < \beta < 1$.
Let $0 < \gamma \in C^1(\overline{\Omega})$ and let $u$ be a weak solution of the equation $\Delta_p^\gamma u = 0$.
Let $\alpha \in \R^d$.

Then
\begin{equation}
\int_\Omega (\doo_\alpha \gamma) |\nabla u|^p \dx{x} = \int_{\doo \Omega} \gamma (\alpha \cdot \nu) |\nabla u|^p \dx{S} - p\int_{\doo \Omega} (\doo_\alpha u) \gamma |\nabla u|^{p-2} \doo_\nu u \dx{S}.
\end{equation}
\end{theorem}
\begin{proof}
Suppose we have the following identity for the solution~$u_\eps$ of the perturbed $p$-Laplace equation~\eqref{eq:eps-laplace}:
\begin{align*}
&\int_\Omega (\doo_\alpha \gamma) \gradeps^{p/2} \dx{x} \\
&= \int_{\doo \Omega} \gamma (\alpha \cdot \nu) \gradeps^{p/2} \dx{S} - p\int_{\doo \Omega} \doo_\alpha u_\eps \gamma \gradeps^\frac{p-2}{2} \doo_\nu u_\eps \dx{S}.
\end{align*}
Then, when $\eps \to 0$, $\nabla u_\eps \to \nabla u$ uniformly (by lemma~\ref{lemma:perturbed_convergence}), so we get the claimed identity.

We now prove the perturbed identity, integrating by parts twice:
\begin{align*}
\int_{\doo \Omega} &(\alpha \cdot \nu) \gamma \gradeps^{p/2} \dx{S} - \int_\Omega (\doo_\alpha \gamma) \left(|\nabla u_\eps|^2 +\eps\right)^{p/2} \dx{x} \\
= & \int_\Omega \gamma \doo_\alpha \left( \gradeps^{p/2} \right) \dx{x} \\
= & p\int_\Omega \gamma \gradeps^\frac{p-2}{2} \nabla u_\eps \cdot \nabla \doo_\alpha u_\eps \dx{x} \\
= & p \int_{\doo \Omega} \doo_\alpha u_\eps \gamma \gradeps^\frac{p-2}{2} \doo_\nu u_\eps \dx{S} \\
&- p\int_\Omega \dive \left( \gamma \gradeps^\frac{p-2}{2} \nabla u_\eps \right) \doo_\alpha u_\eps \dx{x} \\
= &p\int_{\doo \Omega} \doo_\alpha u_\eps \gamma \gradeps^\frac{p-2}{2} \doo_\nu u_\eps \dx{S}.\qedhere
\end{align*}
\end{proof}

We need to recover $\gamma|_{\doo \Omega}$ and $\nabla u|_{\doo \Omega}$ to recover $\int_\Omega \doo_\alpha\gamma |\nabla u|^p \dx{x}$ from the DN~map.

Supposing $\Omega$ is an open set with $C^1$ boundary and $\gamma$ is continuous, we can recover $\gamma(x_0)$ for all boundary points $x_0$.
This follows directly from \cite{Salo:Zhong:2012}.

To calculate $\doo_\nu u$ we use $C^{1,\beta}$ boundary values $v$ so that the solution~$u$ of the boundary value problem~\eqref{eq:p-laplace} is continuously differentiable up to the boundary (by lemma~\ref{lemma:regularity}), so $\doo_\nu u$ is defined pointwise.
Recovering $\doo_\nu u$ from the strong definition of the DN map is straightforward, and we do so in lemma~\ref{lemma:doo_nu}.
However, we do not a~priori have access to the strong DN map, and must recover it from the weak definition~\ref{definition:DN_map}.
This we do in lemma~\ref{lemma:weak_DN_to_strong}, which requires more smoothness on $\Omega$ and boundary values.

We have
\begin{equation}
\Lambda_\gamma^s (v) = \gamma |\nabla u|^{p-2}\doo_\nu u,
\end{equation}
from which by expanding $\nabla u$ we get the equation
\begin{equation*}
\gamma |\nabla_T u + \nabla_\nu u|^{p-2}\doo_\nu u = \Lambda_\gamma^s (v).
\end{equation*}
Taking absolute values and reorganising we have
\begin{equation*}
(|\nabla_T u|^2 + |\nabla_\nu u|^2)^{\frac{p-2}{2}}|\doo_\nu u| = \frac{1}{\gamma}|\Lambda_\gamma^s (v)|.
\end{equation*}
We write the left hand side as a function of $|\doo_\nu u| = |\nabla_\nu u|= t$:
\begin{equation}
F(t) = (|\nabla_T u|^2 + t^2)^{\frac{p-2}{2}}t = \frac{1}{\gamma}|\Lambda_\gamma^s (v)|.
\end{equation}

There is a unique $t_0 \in [0,\infty[$ with $F(t_0) = \frac{1}{\gamma}|\Lambda_\gamma^s (v)|$.
To see this, observe that $\lim_{t \to 0} F(t) = 0$, $\lim_{t \to \infty} F(t) = \infty$, and that the function $F$ is strictly increasing and continuous.

Returning to the original formulation and expanding $|\nabla u|^2$ we therefore get
\begin{equation}
\gamma (|\nabla_T u|^2 + t_0^2)^\frac{p-2}{2} \doo_\nu u = \Lambda_\gamma^s (v),
\end{equation}
from which we can solve $\doo_\nu u$. We have proven the following lemma:
\begin{lemma}\label{lemma:doo_nu}
Suppose $\Omega$ is a bounded open set with $C^{1,\beta}$ boundary.
Let the boundary voltage $v \in C^{1,\beta}$ be known.
Then, for any boundary point $x_0 \in \doo \Omega$, we can compute $\doo_\nu u (x_0)$ from the strong DN map with the following algorithm:
\begin{enumerate}
\item Solve $t$ from the equation \[ (|\nabla_T u (x_0)|^2 + t^2)^{\frac{p-2}{2}}t = \frac{1}{\gamma(x_0)}|\Lambda_\gamma^s (v)(x_0)|. \]
\item Set $\doo_\nu u (x_0) = t\sign(\Lambda_\gamma^s (v)(x_0))$.
\end{enumerate}
\end{lemma}

The following remark is not used in this paper, but in general locating the points where $\nabla u = 0$ is interesting for $p$-harmonic functions $u$.
\begin{remark}
We know the boundary points $x_0$ where $\nabla u(x_0) = 0$, since $\doo_\nu u(x_0) = 0$ if and only if $\Lambda_\gamma^s (v)(x_0) = 0$, and the Dirichlet data determines $\nabla_T u$ at boundary points.
\end{remark}

We still need to recover the strong DN map from the weak one.
\begin{lemma}\label{lemma:weak_DN_to_strong}
Suppose that $\Omega$ has $C^{2,\beta}$ boundary, boundary values~$v \in C^{2,\beta}(\doo \Omega)$ and that $\nabla \gamma$ is H{\"o}lder-continous.
Then we can recover the pointwise values of the strong DN map
\begin{equation}
\gamma(x_0) \left|\nabla u(x_0)\right|^{p-2}\doo_\nu u (x_0)
\end{equation}
from the weak DN map (see definition~\ref{definition:DN_map})
\begin{equation}
\left\langle \Lambda_\gamma^w(v), g \right\rangle = \int_\Omega \gamma |\nabla u|^{p-2} \nabla u \cdot \nabla \tilde g \dx{x}.
\end{equation}
\end{lemma}
\begin{proof}
For $x_0 \in \doo \Omega$ define test functions~$\eta_\delta \colon \R^d \to \R$ by
\begin{equation}
\eta_\delta (x) = c_d \delta^{-d} (\delta - |x-x_0|)\chi_{B(x_0,\delta)}
\end{equation}
where $c_d > 0$ is selected so that $\int_{L} \eta_\delta \dx{x} = 1$,
where $L$ is the tangent space of $\doo \Omega$ at $x_0$.
Then $\eta_\delta$ lies in all first order Sobolev spaces with $|\nabla \eta_\delta(x)| = c_d \delta^{-d}$ in $B(x_0,\delta)$, and $0$ elsewhere.

Suppose $\nabla u(x_0)=0$.
By H{\"o}lder continuity of $|\nabla u|^{p-1}$ we get
\begin{align*}
\left|\langle \Lambda_\gamma^w (v),\eta_\delta\rangle\right| &\lesssim \| |\nabla u|^{p-2} \nabla u \cdot \nabla \eta_\delta \|_{L^1(\Omega \cap B(x_0,\delta))} \\
&\lesssim \| |\nabla u|^{p-1} |\nabla \eta_\delta| \|_{L^1(B(x_0,\delta))} \lesssim \|\delta^{\beta_1} \delta^{-d} \|_{L^1(\Omega \cap B(x_0,\delta))} \\
&\lesssim \delta^{\beta_1 - d} \delta^{d} = \delta^{\beta_1} \to 0 = \Lambda^s_\gamma(x_0)
\end{align*}
as $\delta \to 0$.

Suppose, then, that $\nabla u (x_0) \neq 0$.
Since $\nabla u_\eps \to \nabla u$ as $\eps \to 0$, there is $a > 0$ such that
\begin{equation}
|\nabla u_\eps (x_0)| \geq |\nabla u (x_0)|/2
\end{equation}
holds whenever $a > \eps \geq 0$.

We next prove that the equality
\begin{equation}\label{eq:a}
\gamma(x_0) \left(\left|\nabla u_\eps (x_0)\right|^2 + \eps \right)^\frac{p-2}{2}\doo_\nu u_\eps (x_0) = \lim_{\delta \to 0} \int_\Omega \gamma \gradeps^\frac{p-2}{2} \nabla u_\eps \cdot \nabla \eta_\delta \dx{x}
\end{equation}
holds.
By taking the limit $\eps \to 0$ we get the original claim if the rate of convergence in~\eqref{eq:a} is uniform in $\eps$.
By integrating by parts we get
\begin{align}
\smashoperator{\int_\Omega} \gamma &\gradeps^\frac{p-2}{2} \nabla u_\eps \cdot \nabla \eta_\delta \dx{x} \notag\\
=&-\smashoperator{\int_\Omega} \dive \left( \gamma \gradeps^\frac{p-2}{2}\nabla u_\eps \right) \eta_\delta \dx{x} \notag \\
&+ \smashoperator{\int_{\doo \Omega}} \eta_\delta \gamma \left(\left|\nabla u_\eps \right|^2 + \eps \right)^\frac{p-2}{2}\doo_\nu u_\eps \dx{S} \notag \\
= &\smashoperator{\int_{\doo \Omega \cap B(x_0,\delta)}} \eta_\delta \gamma \left(\left|\nabla u_\eps \right|^2 + \eps \right)^\frac{p-2}{2}\doo_\nu u_\eps \dx{S}, \label{int1}
\end{align}
since $u_\eps$ solves the perturbed equation,
and the integral~\eqref{int1} equals
\begin{align*}
\smashoperator{\int_{\doo \Omega \cap B(x_0,\delta)}} &\eta_\delta \gamma(x_0) \left(\left|\nabla u_\eps(x_0) \right|^2 + \eps \right)^\frac{p-2}{2}\nu(x_0) \cdot \nabla u_\eps(x_0) \dx{S} \\
+& \smashoperator{\int_{\doo \Omega \cap B(x_0,\delta)}} \eta_\delta \bigg( \gamma \left(\left|\nabla u_\eps \right|^2 + \eps \right)^\frac{p-2}{2}\doo_\nu u_\eps \\
&-\gamma(x_0) \left(\left|\nabla u_\eps(x_0) \right|^2 + \eps \right)^\frac{p-2}{2} \doo_\nu u_\eps(x_0) \bigg)\dx{S}\\
=&\left(\gamma(x_0) \left(\left|\nabla u_\eps (x_0)\right|^2 + \eps \right)^\frac{p-2}{2}\doo_\nu u_\eps (x_0) + o_{\delta \to 0}(1)\right)\smashoperator{\int_{\doo \Omega}} \eta_\delta \dx{S},
\end{align*}
since $\gamma$ and $\nabla u_\eps$ are continuous.
By normalisation of the test function this proves the equality~\eqref{eq:a}.

The rate of convergence is uniform in $\eps$ when $\delta$ is so small that $|\nabla u_\eps (x)|$ is bounded away from zero for $x \in B(x_0,\delta)$, since then the following $\eps$-indexed family of functions~\eqref{eq:equi} are equicontinuous with parameter $\eps$:
\begin{equation}\label{eq:equi}
x \mapsto \gamma(x) \left(|\nabla u_\eps(x)|^2 +\eps \right)^{\frac{p-2}{2}}\doo_\nu u_\eps (x).
\end{equation}
The family is H{\"o}lder-continuous on $\doo \Omega$ when $x - x_0$ is so small that $|\nabla u_\eps(x)| \geq m > 0$ for some $m$ and all $\eps < a$:
By lemma~\ref{lemma:eps_regularity} the functions $x \mapsto \nabla u_\eps (x)$ are equicontinuous with parameter $\eps$, and
the functions $w \mapsto \gamma \left( |w|^2 + \eps \right)^\frac{p-2}{2} w \cdot \nu$ are equicontinuous when $|w|$ is bounded from above and away from zero.
\end{proof}

\begin{remark}\label{remark:rellich}
We have shown that
\[ \int_\Omega (\doo_\alpha \gamma) |\nabla u|^p \dx{x} = \int_{\doo \Omega} \gamma (\alpha \cdot \nu) |\nabla u|^p \dx{S} - p\int_{\doo \Omega} (\doo_\alpha u) \Lambda_\gamma^s (u) \dx{S}. \]
Since we know $\gamma|_{\doo \Omega}$ by~\cite[theorem 1.1]{Salo:Zhong:2012}, $\nabla u|_{\doo \Omega}$ by lemma~\ref{lemma:doo_nu} and the strong DN map by lemma~\ref{lemma:weak_DN_to_strong}, we also know $\int_\Omega (\doo_\alpha \gamma) |\nabla u|^p \dx{x}.$
\end{remark}

\begin{remark}
In using the Rellich identity we need to know the values of the conductivity $\gamma$ on the entire boundary $\doo \Omega$.
The gradient $\nabla u_M$ should be large near $x_0$ and small elsewhere (we have not made this precise), so our proof is essentially local in nature.
\end{remark}

The main theorem~\ref{thm:main} now follows, since the value of the integral
\begin{equation}
\int_\Omega \doo_\alpha \gamma |\nabla u_M|^p \dx{x}
\end{equation}
is known (see remark~\ref{remark:rellich}) and it converges to $\doo_\alpha \gamma (x_0)$ by lemma~\ref{lemma:sz_main}.

\section*{Acknowledgements}
The research was partly supported by Academy of Finland.
Part of the research was done during a visit to the Institut Mittag-Leffler (Djursholm, Sweden).
The author is grateful to Mikko Salo and Joonas Ilmavirta for several discussions.

\bibliographystyle{plain}
\bibliography{math}

\end{document}